\documentclass{amsart}

\usepackage{amsfonts,amssymb,amscd,amsmath,enumerate,verbatim,calc,cite}

\newcommand{\topdeg}{\operatorname{topdeg}}

\newcommand{\Potm}{\mathfrak{P}}

\newcommand{\Pol}{\operatorname{Pol}}

\newcommand{\reg}{\operatorname{reg}}
\newcommand{\Tr}{\operatorname{Tr}}

\newcommand{\Quot}{\operatorname{Quot}}

\newcommand{\BN}{{\mathbb N}}

\theoremstyle{plain}
\newtheorem{theorem}{Theorem}
\newtheorem{corollary}[theorem]{Corollary}
\newtheorem{problem}[theorem]{Problem}

\newtheorem{lemma}[theorem]{Lemma}

\newtheorem{proposition}[theorem]{Proposition}

\theoremstyle{definition}

\theoremstyle{remark}

\newcounter{hours}\newcounter{minutes}
\newcommand{\printtime}{%
        \setcounter{hours}{\time/60}%
        \setcounter{minutes}{\time-\value{hours}*60}%
        \thehours\,h\ \theminutes\,min}

\begin{document}

\title[Top Degree of Coinvariants] {On The Top Degree of Coinvariants}
\date{\today,\ \printtime}
\author{ Martin Kohls}
\address{Technische Universit\"at M\"unchen \\
 Zentrum Mathematik-M11\\
Boltzmannstrasse 3\\
 85748 Garching, Germany}
\email{kohls@ma.tum.de}

\author{M\"uf\.it Sezer}
\address { Department of Mathematics, Bilkent University,
 Ankara 06800 Turkey}
\email{sezer@fen.bilkent.edu.tr}
\thanks{We thank T\" {u}bitak for
funding a visit of the first author to Bilkent University, and
Gregor Kemper for funding a visit of the second author to TU
M\"unchen. Second author is also partially supported by
T\"{u}bitak-Tbag/112T113 and T\"{u}ba-Gebip/2010. We also thank M. Domokos for
pointing out that Corollary \ref{DomokosLemma} has already appeared in his paper \cite{DomokosVector}.}

\subjclass[2000]{13A50} \keywords{Coinvariants,
vector invariants}
\begin{abstract}
For a finite group $G$ acting faithfully on a finite dimensional
$F$-vector space $V$, we show that in the  modular case, the top
degree of the vector coinvariants grows unboundedly:
$\lim_{m\to\infty} \topdeg F[V^{m}]_{G}=\infty$. In contrast, in the
non-modular case we identify a situation where the top degree of the vector
coinvariants remains constant. Furthermore, we present a more elementary proof of Steinberg's
theorem  which says that the group order is a lower bound for the
dimension of the coinvariants which is sharp if and only if the
invariant ring is polynomial.
\end{abstract}

 \maketitle

\section{Introduction}
A central problem in invariant theory is to compute the generators
of the invariants of a group action. One crucial element in this
task is determining the degrees of the generators as the knowledge
of these degrees reduces this problem to a problem in a finite
dimensional vector space.  This gives obtaining efficient degree
bounds  a big computational
 significance and research in this direction has always been fashionable since the days
 of Noether to our days, with some
recent spectacular break-throughs, e.g. \cite{MR2811606}.
 Before
we go into more details, we fix our setup. For a shorthand
notion, we will call a finite dimensional representation $V$ of a
finite group $G$ over a field $F$ a \emph{$G$-module}. The action
of $G$ on $V$ induces an action on the symmetric algebra $F[V]=S(V^{*})$  that is given by $\sigma (f)=f\circ\sigma^{-1}$ for
$\sigma \in G$ and $f\in F[V]$. Let $F[V]^{G}$ denote the
corresponding ring of invariants. By a classical theorem of  Noether, it is a finitely
generated algebra, and $\beta(F[V]^{G})$, the Noether number of
the representation, denotes the maximal possible degree of an
indecomposable element, that is, the smallest number $b$ such that
invariants of degree $\le b$ generate the invariant ring. We also
define $\beta (G)=\sup_V \beta(F[V]^{G})$. Another central object
is the Hilbert ideal $I:=F[V]^{G}_{+}F[V]$,  the ideal in $F[V]$
generated by invariants of positive degree. In this paper, we
study the algebra of coinvariants, which is the quotient ring
$F[V]_{G}:=F[V]/I$. This finite dimensional, graded algebra
encodes several interesting properties of the invariant ring and
there has been a fair amount of research on it, see
\cite{MR2852288, MR1424447,DwyerWilkerson, MR2264069,Kane,
KohlsSezerHilbert,SezerCoinvariants,
  MR2193198,  MR1972694,
  SmithForum,   MR0167535} and the
references there. The top degree of the coinvariants, denoted
$\topdeg F[V]_{G}$, is defined to be the largest degree in which
$F[V]_{G}$ is non-zero. This number shares a similar interest for
coinvariants as the Noether number does for invariants.

Equivalently, the top degree can be defined as the smallest number
$d$ such that every monomial $m\in F[V]$ of degree $>d$ is
contained in the Hilbert ideal. Note that this also implies that
the Hilbert ideal is generated by elements of degree at most
$d+1$, a fact which played an important role in the proof of the
Noether bound in the non-modular case. However, it is conjectured
\cite[Conjecture 3.8.6]{MR1918599} that even in the modular case
the group order  is an upper bound for the degrees of the
generators of the Hilbert ideal, which as we will see,  may be
much smaller than the top degree.

 Another natural interpretation comes from
regarding $F[V]$ as a (finite) $F[V]^{G}$-module. Take a minimal
set of homogeneous module generators $g_{i}$ of $F[V]$ over
$F[V]^{G}$, so $F[V]=\sum_{i=1}^{t}F[V]^{G}g_{i}$. From the graded
Nakayama lemma, it follows that the top degree $d$ equals the
maximum  of the degrees of the generators, and the number of
generators equals the dimension of the coinvariants as a vector
space.

Recall that the transfer of  $f\in F[V]$ is defined by
$\Tr(f)=\sum_{\sigma\in
  G}\sigma(f)$. Another important application of the top degree is that in the
modular case, it yields an upper bound for the maximal degree of
an indecomposable transfer:  Take $f\in F[V]$ homogeneous. Then we
can write $f=\sum_{i=1}^{t}h_{i}g_{i}$  with homogeneous
invariants $h_{i}$ and module generators $g_{i}$ as above.
Therefore, $\Tr(f)=\sum_{i=1}^{t}h_{i}\Tr(g_{i})$. Assume
$\deg(f)$ is bigger than the top degree of $F[V]_{G}$. Then all
$h_{i}$'s are zero or of positive degree. We are done if also all
$\Tr(g_{i})$'s are zero or of positive degree. Note that one of
the module generators, say $g_{1}$, is a constant.  Since we are
in the modular case we have $\Tr(1)=|G|\cdot 1=0$, so we are done.
Knowing the maximal degree of an indecomposable transfer has been
very critical so far, since in almost all modular cases where the
Noether number is known, there is an indecomposable transfer of
degree equal to the Noether number, see \cite{MR2264069}. In the
non-modular case (i.e. the characteristic of $F$ does not divide
the group order $|G|$), the invariant ring is generated by transfers
and so a bound for the degree of an indecomposable transfer is a
bound for the Noether number.  Since $\Tr(1)\ne 0$,  the argument
above does not carry over to this characteristic. Nevertheless, in
the non-modular case, the top degree plus one is   an upper bound
for the Noether number and this bound is sharp:  The
Noether number corresponding to the natural action of $S_{2}$ on
$F[x_{1},x_{2}]$ is two while the top degree of the coinvariants
is one.

We now give an outline of the paper. Section \ref{topDegNonMod} is
mainly concerned with  the non-modular case, where we collect some
 consequences of previous work on the top degree of coinvariants.
Most notably, a quite recent result of Cziszter and Domokos
implies that for a given non-modular group $G$, the maximal
possible top degree equals the maximal possible Noether number
minus one. In particular, $|G|-1$ gives an upper bound for the top
degree.

In contrast, we show in section \ref{TopDegMod} that for a given
faithful modular representation $V$, the top degree of the vector
coinvariants $F[V^{m}]_{G}$ grows unboundedly with $m$. This also
fits nicely with a result of Richman \cite{MR1423197} which
asserts the similar behavior for the Noether number of the vector
invariants $F[V^{m}]^{G}$.

In the following section \ref{topdegSymmetric} we consider a non-modular situation where the lead term ideal
of $F[V]^{G}_{+}F[V]$ is generated by pure powers of the variables. In this case we show that the top degree of
the vector coinvariants $F[V^{m}]_G$ is constant. This way, for the  natural
action of the symmetric group $S_n$ on a polynomial ring with $n$ variables
we get a new proof that the top degree of any of the vector coinvariants of this action is  ${n\choose 2}$.

In section \ref{Steinberg} we will give a new elementary proof of
Steinberg's celebrated theorem which states that the group order
is a lower bound for the dimension of the coinvariants with
equality holding if and only if the invariant ring is polynomial.

\section{Top degree  in the non-modular case}\label{topDegNonMod}

In this section we note several facts about the top degree of
coinvariants in the non-modular case, which are a bit spread out
in the literature. Although these statements follow rather quickly
from previous results,  it seems that the statements themselves
have not been formulated in terms of coinvariants before. Using a
very recent result of Cziszter and Domokos \cite{DomokosDavenport}
 we obtain  in Theorem \ref{Noether} that the supremum of the top degrees of coinvariants is one less than the Noether number  of the group.
Since the Noether number is bounded by the group order, we establish $|G|-1$ as an upper bound for the top degree of coinvariants of
any non-modular representation. This upper bound also follows directly from
Fogarty's proof of the Noether bound.  We take the crucial part of this proof here as Lemma \ref{fogarty}. Using this lemma we also obtain a relative bound
for the top degree of  coinvariants, see Proposition \ref{relGh}.
We end this section with a brief discussion of the relation between the
Davenport constant and the top degree in the abelian group case.

\begin{theorem}\label{Noether}
Assume that the characteristic of $F$ does not divide the group
order $|G|$. Then for any $G$-module $V$, we have
\[
\beta(F[V]^{G})\le \topdeg F[V]_{G}+1\le \beta(G)\le|G|.
\]
In particular, we have that
\[
\topdeg(G)+1:=\sup_V \topdeg F[V]_{G} +1=\beta(G).
\]
\end{theorem}

\begin{proof}
Let $I$ denote the Hilbert ideal of $F[V]^{G}$ and $d$ denote the
top degree of $F[V]_{G}$. As mentioned in the introduction  $I$ is
generated by elements of degree at most $d+1$. As we are in the
non-modular case, this implies that $F[V]^{G}$ is generated by
invariants of degree at most $d+1$, which proves the first
inequality. By \cite[Lemma 3.1]{DomokosDavenport}, for any
$G$-module $V$ there exists an irreducible $G$-module $U$ such
that $\topdeg F[V]_{G}+1\le \beta(F[V\oplus U]^{G})$, which proves
the second inequality. Finally, the Noether number is at most the
group order in the non-modular case, see \cite{FleischmannNoether,
FogartyNoether}. Now the last statement follows from choosing a
$G$-module $V$ with $\beta(F[V]^{G})=\beta(G)$.
\end{proof}

There are many bounds on $\beta(G)$ in invariant theory
literature.  By this theorem, they translate into bounds for
$\topdeg(G)+1$. For example, if $H$ is a normal subgroup of $G$,
in the non-modular case we have $\beta(G)\le\beta(H)\beta(G/H)$
\cite[(3.1)]{Lempken}. So we get $\topdeg (G)+1\le (\topdeg
(G/H)+1)(\topdeg (H)+1)$.

However, for a given module $V$, its
Noether number can be much smaller than the top degree. For
example, for the natural action of the symmetric group on $n$
variables, the invariants have Noether number $n$, while the top
degree of the coinvariants is ${n\choose 2}$.

A key step in Fogarty's proof of $\beta(G)\le|G|$ in the non-modular case is the following lemma \cite{FogartyNoether}.

\begin{lemma}[{See \cite[Lemma
 3.8.1]{MR1918599}}]
\label{fogarty} Let $A$ be a commutative ring with identity, $G$ a
finite group of automorphisms of $A$, and $J\subseteq A$ a
$G$-stable ideal. If the order of $G$ is invertible in $A$, then
$J^{|G|}\subseteq J^{G}A$.
\end{lemma}

This lemma also yields a relative bound for the top degree of  coinvariants.

\begin{proposition}\label{relGh}
Assume $H$ is a normal subgroup of $G$ and the characteristic of $F$ does not
divide the index  $(G:H)$. Then we have the inequality
\[
\topdeg (F[V]_{G})+1 \le (G:H)(\topdeg(F[V]_{H})+1).
\]
\end{proposition}

\begin{proof}
Let $m$ denote the top degree of $F[V]_H$, and $d$ denote the index $(G:H)$. Then all monomials of
degree $m+1$ of $F[V]$ lie in $I:=F[V]^H_+\cdot F[V]$. Therefore all
monomials of degree $d(m+1)$ lie in
\[
I^{d}=(F[V]^H_+\cdot F[V])^{d}=(F[V]^H_+)^{d}\cdot F[V].
\]
By the previous lemma, applied to the group $(G/H)$ acting on $A=F[V]^{H}$ and the $G/H$-stable
ideal $J=F[V]^{H}_{+}$, we have \[(F[V]^H_+)^{d}\subseteq
(F[V]^{H}_{+})^{G/H}F[V]^{H}=F[V]^{G}_{+}F[V]^{H}\subseteq F[V]^{G}_{+}F[V].\] Therefore, all
monomials of degree $d(m+1)$ lie in $F[V]^{G}_{+}F[V]$.
\end{proof}

For abelian groups, the  top degree of the coinvariants has
another interpretation in terms of the \emph{Davenport constant}
of the group. We conclude with a discussion on this relation. For
the rest of this section assume that  $G$ is an abelian group with
$|G|\in F^*$. Since extending
 the ground field does not change the
top degree of coinvariants we  assume that $F$ is algebraically
closed. In this case the action is diagonalizable so we may as
well assume that  $F[V]=F[x_1, \dots , x_n]$, where $x_1, \dots ,
x_n$ is a basis of $V^*$ on which $G$ acts diagonally.
 For each
$1\le i\le n$, let $\kappa_i$ denote the character corresponding
to the action on $x_i$. Then a monomial $x_1^{a_1}\cdots
x_n^{a_n}$ is in $F[V]^G$ if $\sum_{1\le i\le n}a_i\kappa_i=0$.
Moreover, a monomial $x_1^{a_1}\cdots x_n^{a_n}$ is in the Hilbert ideal $I$
 if it is divisible by an invariant monomial,
that is there exist integers $0\le b_i\le a_i$ such that
$\sum_{1\le i\le n}b_i\kappa_i=0$. For an abelian group $G$, let
$S(G)$ denote the minimal integer such that every set of elements,
with repetitions allowed, of size $S(G)$ in $G$ has a subsequence
that sums up to zero. It also equals the length of the longest
non-shortenable zero sum of elements (with repetitions) of $G$.
This number is called the Davenport constant of $G$. Since the
character group of $G$ is isomorphic to $G$ it follows that every
monomial in $F[V]$ of degree $S(G)$ lies in $I$. This gives
$\topdeg (G)+1\le S(G)$. On the other hand by constructing an
action using the characters in the longest sequence of  elements
with no subsequence summing up to zero we get a $G$ module $V$
with $\topdeg F[V]_G+1=S(G)$. Similarly one can show that
$\beta(G)=S(G)$, see also \cite[Proposition 2.2]{Schmid}.  So it
follows that
 $$S(G)=\topdeg (G)+1=\beta(G).$$

Results on the Davenport constant  therefore apply to the top
degree of the coinvariants, and vice versa. See
\cite{GaoGeroldinger} for a survey on the Davenport constant.
Here we just quote two famous results due to Olson \cite{Olson1,
Olson2}: If $Z_{n}$ denotes the cyclic group of order $n$, then if
$a|b$, we have $S(Z_{a}\times Z_{b})=a+b-1$. If $p$ is a prime,
then $S(Z_{p^{d_{1}}}\times\dots\times
Z_{p^{d_{r}}})=1+\sum_{i=1}^{r}(p^{d_{i}}-1)$.

\section{The unboundedness of the top degree for modular coinvariants}\label{TopDegMod}

In this section we specialize to the modular case and show that,
in contrast to the non-modular case, the top degree of the
coinvariants of a given group  can become arbitrarily large. We
start with a collection of observations which despite their
simplicity give useful upper and lower bounds.

\begin{lemma}\label{subgroups}
Let $H$ be a subgroup of $G$ and $V$ be a $G$-module. Then
\[
\topdeg F[V]_{H}\le\topdeg F[V]_{G} \quad \text{ and }\quad \dim F[V]_{H}\le
\dim F[V]_{G}.
\]
\end{lemma}
\begin{proof}
The inclusion $F[V]^{G}_{+}\subseteq F[V]^{H}_{+}$ induces a degree preserving
surjection
\[F[V]_{G}=F[V]/F[V]_{+}^{G}F[V] \twoheadrightarrow F[V]/F[V]_{+}^{H}F[V] =F[V]_{H},\]
which immediately establishes the claim.
\end{proof}

\begin{lemma}\label{submods}
Let  $U$ be a $G$-submodule of $V$. Then
\[
\topdeg F[U]_{G} \le \topdeg F[V]_{G}\quad\text{ and }\quad \dim F[U]_{G}\le
\dim F[V]_{G}.
\]
\end{lemma}
\begin{proof}
The inclusion $U\subseteq V$ induces the epimorphism
\[
\varphi: F[V]\twoheadrightarrow F[U],\quad\quad f\mapsto f|_{U},
\]
which restricts to a (generally non-surjective) morphism $F[V]^{G}\rightarrow
F[U]^{G}$. We therefore get a degree preserving epimorphism
\[
\overline{\varphi}: \,\, F[V]_{G}=F[V]/F[V]_{+}^{G}F[V]\twoheadrightarrow F[U]/F[U]_{+}^{G}F[U]=F[U]_{G},
\]
which yields both inequalities.
\end{proof}

For a  $G$-module $V$, let $V^m$ denote the $m$-fold direct sum of $V$.
\begin{lemma}
For any two $G$-modules $V$ and $W$ we have,
\[
\topdeg F[V\oplus W]_{G}\le \topdeg F[V]_{G}+\topdeg F[W]_{G}.
\]
In particular, we have $\topdeg F[V^{m}]_{G}\le m \topdeg F[V]_{G}$ for all $m\in\BN$.
\end{lemma}

\begin{proof}
Assume that $M\in F[V\oplus W]$ is a monomial of degree at least $\topdeg F[V]_{G}+\topdeg
F[W]_{G}+1$. Write $M=M'M''$ with $M'\in F[V]$ and $M''\in
F[W]$.   Then we have either $\deg M'> \topdeg F[V]_{G}$ or $\deg M''> \topdeg
F[W]_{G}$. Without loss of generality we assume the former inequality. Then $M'\in
F[V]_{+}^{G}F[V]$, which implies $M\in F[V]_{+}^{G}F[V\oplus W]\subseteq F[V\oplus
W]^{G}_{+}F[V\oplus W]$.
\end{proof}

Let $V_{\reg}:=FG$ denote the regular representation of $G$. For
any $G$-module $V$, we have an embedding $V\hookrightarrow
V_{\reg}^{\dim_{F}(V)}$. Thus we get the following as  a corollary
to the preceding lemmas.

\begin{corollary}
For any $G$-module $V$, we have
\[
\topdeg F[V]_{G}\le \dim_{F}(V)\topdeg F[V_{\reg}]_{G}.
\]
\end{corollary}

In view of Theorem \ref{Noether}, the main result of this section
nicely separates the modular coinvariants from  the non-modular
ones.

\begin{theorem}\label{topdegCoinvariantsModular}
Let $V$ be a faithful $G$-module and assume that the characteristic $p>0$ of $F$ divides the group order
$|G|$. Then
\[
\lim_{m\to\infty} \topdeg F[V^{m}]_{G}=\infty.
\]
\end{theorem}
\begin{proof}
 Pick a subgroup $H$ of $G$  of size $p$. It is well
known that the indecomposable $H$-modules consist
of modules $V_{k}$ for $1\le k\le p$, where $V_k$ is a $k$-dimensional vector space on which a generator of $H$ acts via a single Jordan block
with ones on the diagonal.
 Therefore, as an $H$-module, $V$ decomposes in a direct sum
$V=\bigoplus_{i=1}^{q}V_{k_{i}}$. Note that $V$ is also faithful
as an $H$-module, so without loss of generality we  assume
$k_{1}\ge 2$. Notice that we have  an $H$-module inclusion
$V_{k}\subseteq V_{l}$  for any pair of integers $1\le k\le l\le
p$. In particular, we have the $H$-module inclusions
\[
V_{2}\subseteq V_{k_{1}} \subseteq \bigoplus_{i=1}^{q}V_{k_{i}}=V.
\]
Therefore for any $m\in\BN$ we have  $V_{2}^{m}\subseteq
V^{m}$ as $H$-modules. We now get 
\[
\topdeg F[V^{m}]_{G}\ge \topdeg F[V^{m}]_{H}
\]
by Lemma \ref{subgroups}, and furthermore
\[
\topdeg F[V^{m}]_{H}\ge \topdeg F[V_{2}^{m}]_{H}
\]
by Lemma \ref{submods}. On the other hand from  \cite[Theorem 2.1]{MR2193198}  we get $\topdeg
F[V_{2}^{m}]_{H}=m(p-1)$. So it follows that
\[
\topdeg F[V^{m}]_{G}\ge m(p-1) \text{ for all }m\in\BN.
\]
\end{proof}

We will show next that the dimensions of the vector coinvariants
always grow unboundedly as well, even in the non-modular case. We
start again with a simple but useful observation:

\begin{lemma}\label{dimBiggerTopDeg}
For any $G$-module $V$, we have
\[
\dim F[V]_{G}\ge \topdeg F[V]_{G}+1.
\]
\end{lemma}

\begin{proof}
If $d$ is the top degree of $F[V]_{G}$, then there exists a
monomial $m$ of degree $d$ which is not in the Hilbert ideal $I$.
Then every divisor of $m$ is also not contained in $I$, which
means that $F[V]_{G}$ contains a non-zero class in each degree
$\le d$. As elements of different degrees are linearly
independent, this finishes the proof.
\end{proof}

\begin{proposition}
For any non-trivial $G$-module $V$, we have
\[
\lim_{m\to\infty} \dim F[V^{m}]_{G}=\infty.
\]
\end{proposition}

\begin{proof}
We can assume the action of $G$ is faithful. In the modular case,
the result follows from Lemma \ref{dimBiggerTopDeg} and Theorem
\ref{topdegCoinvariantsModular}. In the non-modular case, choose a
subgroup $H=\langle \sigma\rangle$ of $G$ of prime order $q$,
which is  coprime to the characteristic of $F$. Choose a basis
$x_1, \dots , x_n$ of $V^*$  on which  $\sigma$ acts diagonally.
Since $V$ is a faithful $H$-module as well, we may assume that the
action of $\sigma$ on $x_1$ is given by  multiplication with a
primitive $q$th root of unity. Let $x_{1,1},\ldots,x_{1,m}$ denote
the copies of $x_{1}$ in $F[V^{m}]$. Then none of the linear
combinations of these variables lie in the Hilbert ideal
$F[V^{m}]^{H}_+F[V^{m}]$, so they form an independent set of
classes in $F[V^{m}]_{H}$. Therefore by Lemma \ref{subgroups} we
have\[ \dim F[V^{m}]_{G}\ge \dim F[V^{m}]_{H}\ge m,
\]
which clearly establishes the claim.
\end{proof}

\section{Top degree of vector coinvariants in the non-modular case}\label{topdegSymmetric}

In this section we study vector copies of an action of a group in
the non-modular case. Obtaining generating  invariants for these
actions is  generally a difficult problem nevertheless the degrees
of polynomials in minimal generating sets do not change in many
cases, see \cite[Example 3.10]{MR2414957} for a rare example. Our
computer aided search of  examples indicate that many classes of
coinvariants enjoy a similar type of saturation.
 We note this as a problem for future study.
\begin{problem}
\label{proble}
 Assume that $V$ is a non-modular $G$-module. Prove or disprove that
$$\topdeg F[V^{m}]_{G}=\topdeg F[V]_{G}$$ for any positive integer $m$. Find classes of groups and modules for which the equality is true.
\end{problem}

 We prove the equality  above for a certain special case. First, we  review the concept of polarization as we use polarized polynomials
in our computations. Let $V$ be a non-modular $G$-module and set
$A:=F[V]=F[x_{1},\ldots,x_{n}]$ and
$B:=F[V^{m}]=F[x_{1,1},\ldots,x_{n,1},\ldots,x_{1,m},\ldots,x_{n,m}]$.
We use the lexicographic order on $B$ such that
\[
x_{1,1}> x_{1,2} >\ldots >x_{1,m}> \ldots > x_{n,1}>\ldots > x_{n,m}
\]
and the order on $A$ is obtained by setting $m=1$. For an ideal $I$ in $A$ or
$B$ we denote the lead term ideal of $I$ with $L(I)$. Also $L(f)$
 denotes the lead term of  a polynomial $f$
in these rings. We introduce additional variables  $t_{1},\ldots,t_{m}$ and  define an algebra homomorphism
\[
\phi: A\rightarrow B[t_{1},\ldots,t_{m}],\quad x_{i}\mapsto
x_{i,1}t_{1}+\ldots+x_{i,m}t_{m}.
\] Then for any $f\in A$,
 write
\[
\phi(f)=\sum_{i_{1},\ldots,i_{m}}f_{i_{1},\ldots,i_{m}}t_{1}^{i_{1}}\dots
t_{m}^{i_{m}},
\]
where $f_{i_{1},\ldots,i_{m}}\in B$.
This process is called polarization and we let  $\Pol (f)$ denote the set of coefficients $\phi_{i_{1},\ldots,i_{m}}(f):=f_{i_{1},\ldots,i_{m}}$ of $\phi(f)$. Restricting
 to invariants, it is well known that we get a map
 $\Pol: A^{G}\rightarrow \Potm(B^{G})$, where $\Potm (B^{G})$ denotes the power set of $B^{G}$.
Let $I_{A}:=A^{G}_{+}A$ denote the Hilbert ideal of
$A$, and similarly $I_{B}$ denote the Hilbert ideal of $B$. We show that polarization of a polynomial in $I_A$ gives polynomials in $I_B$.

\begin{lemma}
Let $f\in I_A$.
Then $\Pol (f)\in \Potm (I_B)$.
\end{lemma}

\begin{proof}
Since each $\phi_{i_{1},\ldots,i_{m}}$ is a linear map, we may take $f=hg$ with
$h\in A^G_+$  and $g\in A$. Write
$\phi (h)=\sum_{j_{1},\ldots,j_{m}}h_{j_{1},\ldots,j_{m}}t_{1}^{j_{1}}\dots
t_{m}^{j_{m}}$ and  $\phi (g)=\sum_{q_{1},\ldots,q_{m}}g_{q_{1},\ldots,q_{m}}t_{1}^{q_{1}}\dots
t_{m}^{q_{m}}$. Note that  we have $h_{j_{1},\ldots,j_{m}}\in B^G_+$ since polarization preserves degrees. It follows that
$$f_{i_{1},\ldots,i_{m}}=\sum_{j_k+q_k=i_k, \; 1\le k\le
  m}h_{j_{1},\ldots,j_{m}}g_{q_{1},\ldots,q_{m}}\in B^{G}_{+}B,$$
which proves the lemma.
\end{proof}

We now identify a situation where the  equality in Problem
\ref{proble} holds.

\begin{theorem}
Let $F$ be a field of characteristic $p$ and $V$ a $G$-module. Assume that
there exist integers $a_1, \dots ,a_n$, strictly smaller than $p$  in case of positive
characteristic, such that
$L(I_A)=(x_1^{a_1}, \dots ,x_n^{a_n})$.   Then
we have
\[
\topdeg F[V^{m}]_G=\topdeg F[V]_G=\sum_{i=1}^{n}(a_{i}-1) \quad\quad\text{ for all }m\in\BN.
\]
\end{theorem}

\begin{proof}
Since the monomials in $A$ that do not lie in $L(I_A)$ form a vector space basis for $F[V]_G$, we have
$\topdeg F[V]_G=\sum_{i=1}^{n}(a_i-1)$. From Lemma \ref{submods} we also have
$\topdeg
F[V]_{G}\le \topdeg F[V^{m}]_{G}$. Therefore, to prove the  theorem it suffices to show
  $\topdeg F[V^{m}]_{G}\le \sum_{i=1}^{n}(a_i-1)$. To this end we demonstrate that
the lead term ideal
 $L(I_{B})$ contains  the set
\[
S:=\{x_{i,1}^{a_{i,1}}x_{i,2}^{a_{i,2}}\cdot\ldots\cdot x_{i,m}^{a_{i,m}}|\quad
i=1,\ldots,n,\,\,\,\,\, a_{i,1}+\ldots+a_{i,m}=a_i\}.
\]
Take a homogeneous element $f\in I_{A}$ with  $L(f)=x_{i}^{a_i}$. So $f=x_{i}^{a_{i}}+h$ where each term in $h$ is strictly
lex-smaller than $x_{i}^{a_i}$. Then each term of $h$ is of
the form $x_{i}^{b_{i}}x_{i+1}^{b_{i+1}}\dots x_{n}^{b_{n}}$ with
$b_{i}<a_i$. Considering
\[
\phi(x_{i}^{a_i})=(t_{1}x_{i,1}+\ldots+t_{m}x_{i,m})^{a_i}
\]
and
\[
\phi(x_{i}^{b_{i}}x_{i+1}^{b_{i+1}}\dots x_{n}^{b_{n}})=   (t_{1}x_{i,1}+\ldots+t_{m}x_{i,m})^{b_{i}}\dots (t_{1}x_{n,1}+\ldots+t_{m}x_{n,m})^{b_{n}},
\]
we get by the choice of our order that, for any sequence $a_{i,1},\ldots,a_{i,m}\in\BN_{0}$  satisfying
$a_{i,1}+\ldots+a_{i,m}=a_i$ we have
\[
L(\phi_{a_{i,1},\ldots,a_{i,m}}(f))=L(\phi_{a_{i,1},\ldots,a_{i,m}}(x_{i}^{a_i}))=\frac{{a_i}!}{a_{i,1}!\dots
a_{i,m}!}x_{i,1}^{a_{i,1}}x_{i,2}^{a_{i,2}}\cdot\ldots\cdot
x_{i,m}^{a_{i,m}}.
\]
As for positive characteristic $p$, $a_{i}$ is strictly smaller
than $p$ by hypothesis, so the coefficient is nonzero. Moreover,
$\phi_{a_{i,1},\ldots,a_{i,m}}(f)\in I_B$ by the
 previous lemma. This finishes the proof.
\end{proof}

 Consider the natural action of the symmetric group $S_n$ on $F[V]$. It is well known that
$L(I_A)=(x_1, x_2^2, \dots , x_n^n )$, see for example  \cite[Proposition 1.1]{MR2246711}.
So the theorem applies and we get the following corollary, which also appears
as the special case $q=1$ in \cite[Lemma 3.1]{DomokosVector}.

\begin{corollary}\label{DomokosLemma}
 Let $F$ be a field of characteristic $p$ and $V$ be the natural
 $S_n$-module. If $p=0$ or $p>n$, then
for any positive integer $m$ we have $$\topdeg F[V^{m}]_{S_n}={n \choose 2}.$$
 \end{corollary}
We want to emphasize here again the sharp contrast to the case $0< p\le n$, where by
Theorem \ref{topdegCoinvariantsModular} we have $\lim_{m\to\infty} \topdeg F[V^{m}]_{S_{n}}=\infty$.

\section{A new proof for Steinberg's Theorem}\label{Steinberg}

The following might be one of the  most celebrated results on coinvariants.

\begin{theorem}[Steinberg]
For any faithful $G$-module $V$, we have
\[
|G|\le \dim F[V]_{G}
\]
with equality  if and only if $F[V]^{G}$ is  polynomial.
\end{theorem}

Note that by the famous Chevalley-Shephard-Todd-Serre-Theorem,
$F[V]^{G}$ being polynomial always implies $G$ is a reflection
group, and in the non-modular case  the converse is also true.
Steinberg \cite{MR0167535} proves the  theorem above for the
complex numbers using analysis. More recently, Smith
\cite{MR1972694} generalized the theorem to arbitrary fields,
using  some heavy machinery from homological algebra. We now give
an almost elementary proof.

\begin{proof}
The group $G$ acts naturally on the quotient field $F(V)$, hence by Galois
theory we have $\dim_{F(V)^{G}}F(V)=|G|$.  Let $S$ be a minimal generating set for $F[V]$ as a module over $F[V]^G$.
Then by the graded Nakayama lemma, $S$ projects injectively onto a vector space basis for $F[V]_G$. Moreover, from Proposition  \ref{son} we get that $S$
also generates $F(V)$ as an $F(V)^{G}$-vector
space. So we have
 \[
\dim F[V]_{G}=|S|\ge\dim_{F(G)^{G}}F(V)=|G|.
\]

If equality holds, then $S$ is a basis for $F(V)$ over $F(V)^{G}$, so it is
$F(V)^{G}$- and hence $F[V]^{G}$-linearly independent. This implies that $F[V]$
 is a free $F[V]^{G}$-module. Now by \cite[Corollary 6.2.3]{MR1249931}, we
 get that $F[V]^{G}$ is  polynomial. The reverse implication is
 straightforward: If $F[V]^{G}$ is polynomially generated by invariants of
 degree $d_{1},\ldots,d_{n}$, the Cohen-Macaulayness of $F[V]$ implies that
 $F[V]$ is freely generated over $F[V]^{G}$ by $d_{1}\cdot\ldots \cdot d_{n}$ many
 generators, and it is well known that this product equals $|G|$, see Smith's
 proof \cite{MR1972694}
for the details.
\end{proof}

Above we used the following well-known proposition. We give a
proof here due to lack of reference. Let $\Quot (D)$ denote the
quotient field of an integral domain~$D$.

\begin{proposition}
\label{son}
Assume $A\subseteq R$ is an integral extension of integral domains. Then
\[
\Quot(R)=\left\{\frac{r}{a}|\quad r\in R,\,\,a\in A\setminus\{0\}\right\}=(A\setminus\{0\})^{-1}R.
\]
In particular, if  $S\subseteq R$
generates $R$ as an $A$-module,  then $S$ generates $\Quot(R)$ as a
$\Quot(A)$-vector space.
\end{proposition}

\begin{proof}
Assume $\frac{f}{g}\in\Quot(R)$ with $f,g\in R$ and  $g\neq 0$.  Let
\[
g^{t}+a_{t-1}g_{t-1}+\ldots+a_{1}g+a_{0}=0
\]
be a monic equation of minimal degree satisfied by $g$.
Then $a_{0}\ne 0$ and dividing this equation
by $g$ shows $\frac{a_{0}}{g}\in R$. Therefore $\frac{f}{g}=\frac{f}{a_{0}}\cdot
\frac{a_{0}}{g}\in (A\setminus\{0\})^{-1}R$.
\end{proof}

\bibliographystyle{plain}
\bibliography{OurBib}
\end{document}